\documentclass{amsart}
\usepackage{latexsym}
\usepackage{amstext}
\usepackage{amsmath}
\usepackage{amssymb}
\usepackage{amsthm}
\usepackage{enumitem}
\usepackage{mathrsfs} 
\usepackage{url}
\usepackage{comment}
\usepackage[utf8]{inputenc}
\usepackage[T1]{fontenc}
\usepackage{amsmath,amssymb,amsthm,mathtools}
\usepackage{hyperref}

\theoremstyle{plain}
\newtheorem{theorem}{Theorem}[section]

\newtheorem{corollary}[theorem]{Corollary}
\newtheorem{lemma}[theorem]{Lemma}

\theoremstyle{definition}
\newtheorem{definition}[theorem]{Definition}

\newtheorem{remark}[theorem]{Remark}

\def \O {\mathcal{O}}
\def \II {\mathcal{I}}

\newcommand{\cR}{\mathcal{R}}
\newcommand{\cQ}{\mathcal{Q}}

\newcommand{\cL}{\mathcal{L}}

\DeclareMathOperator{\Supp}{Supp}

\DeclareMathOperator{\codim}{codim}
\DeclareMathOperator{\Proj}{Proj}

\allowdisplaybreaks
\title[Schmidt's subspace theorem for moving targets]{On a theorem of Schmidt's subspace theory with moving targets}
\author{GuanHeng Zhao, YuXi Li}

\begin{document}
	\begin{abstract}
		The Schmidt's subspace theory with moving targets, as a significant branch in this field, has been substantially developed in recent years. We continue the approach of the previous work, construct a weighted version of generalized Schmidt subspace theory with moving targets.
		As a supplement and sharp extension to the relevant results.
	\end{abstract}	

	\maketitle	
	\tableofcontents
	\section{Preliminary knowledge}
	The symbolic system employed in this paper is relatively intricate. To avoid ambiguity and to ensure a precise understanding of the theoretical framework, we will first introduce relevant  definitions and illustrations.
	\subsection{Cardinal conditions}
	Let $k$ be a number field. Denote by $M_{k}$ the set of places (equivalence classes of absolute values) of $k$ and by $M_{k}^{\infty}$ the set of archimedean places of $k$. For each $v \in M_{k}$, we choose the normalized absolute value $|\cdot|_{v}$ such that $|\cdot|_{v}=|\cdot|$ on $\mathbb{Q}$ (the standard absolute value) if $v$ is archimedean, and $|p|_{v}=p^{-1}$ if $v$ is non-archimedean and lies above the rational prime $p$. For each $v \in M_{k}$, denote by $k_{v}$ the completion of $k$ with respect to $v$ and set
	$$
	n_{v}:=\left[k_{v}: \mathbb{Q}_{v}\right] /[k: \mathbb{Q}] .
	$$
	We put $\|x\|_{v}=|x|_{v}^{n_{v}}$. The product formula is stated as follows
	$$
	\prod_{v \in M_{k}}\|x\|_{v}=1, \text { for } x \in k^{*}
	$$
	For $x=\left(x_{0}, \ldots, x_{M}\right) \in k^{M+1}$, define
	$$
	\|x\|_{v}:=\max \left\{\left\|x_{0}\right\|_{v}, \ldots,\left\|x_{M}\right\|_{v}\right\}, v \in M_{k} .
	$$
	\subsection{Height}
	We define the absolute logarithmic height of a projective point $\mathbf{x}=\left(x_{0}: \cdots: x_{M}\right) \in \mathbb{P}^{M}(k)$ by
	$$
	h(\mathbf{x}):=\sum_{v \in M_{k}} \log \|\mathbf{x}\|_{v} .
	$$
	By the product formula, this does not depend on the choice of homogeneous coordinates $\left(x_{0}: \cdots: x_{M}\right)$ of $\mathbf{x}$. If $\mathbf{x} \in k^{*}$, we define the absolute logarithmic height of $\mathbf{x}$ by
	$$
	h(\mathbf{x}):=\sum_{v \in M_{k}} \log ^{+}\|\mathbf{x}\|_{v},
	$$
	where $\log ^{+} a=\log \max \{1, a\}$.
	Let $\mathbb{N}=\{0,1,2, \ldots\}$ and for a positive integer $d$, we set
	$$
	\mathcal{T}_{d}:=\left\{\left(i_{0}, \ldots, i_{M}\right) \in \mathbb{N}^{M+1}: i_{0}+\cdots+i_{M}=d\right\} .
	$$
	Let $Q=\sum_{I \in \mathcal{T}_{d}} a_{I} \mathbf{x}^{I}$ be a homogeneous polynomial of degree $d$ in $k\left[x_{0}, \ldots, x_{M}\right]$, where $\mathbf{x}^{I}=x_{0}^{i_{0}} \ldots x_{M}^{i_{M}}$ for $\mathbf{x}=\left(x_{0}, \ldots, x_{M}\right)$ and $I=\left(i_{0}, \ldots, i_{M}\right)$. Define $\|Q\|_{v}=\max \left\{\left\|a_{I}\right\|_{v} ; I \in\right.$ $\left.\mathcal{T}_{d}\right\}$. The height of $Q$ is defined by
	$$
	h(Q)=\sum_{v \in M_{k}} \log \|Q\|_{v} .
	$$
	\subsection{Weil function (local height function)}
	For each $v \in M_{k}$, we define the Weil function $\lambda_{Q, v}$ by
	$$
	\lambda_{Q, v}(\mathbf{x}):=\log \frac{\|\mathbf{x}\|_{v}^{d} \cdot\|Q\|_{v}}{\|Q(\mathbf{x})\|_{v}}, \mathbf{x} \in \mathbb{P}^{M}(k) \backslash\{Q=0\} .
	$$
	To better interpret the special properties of Weil function, we consider a higher-level research object, closed subschemes in a projective algebraic variety $V$ (which also apply to the above definition via its inherent concept). It satisfy the following properties according to \cite{HL21} : if $Y$ and $Z$ are two closed subschemes of $V$, defined over $k$, and $v$ is a place of $k$, then up to $O(1)$,
	$$\lambda_{Y\cap Z,v}=\min\{ \lambda_{Y,v},\lambda_{Z,v}\}, \lambda_{Y+Z,v}=\lambda_{Y,v}+\lambda_{Z,v}, \lambda_{Y,v}\leq \lambda_{Z,v}, \text{ if }Y\subset Z.$$
	The formulas $Y\cap Z$, $Y+Z$, $Y\subset Z$ above are defined in terms of the associated ideal sheaves. 
	Furthermore, for $c_1,\ldots,c_r$ are positive integers, it is generally true that if $Y_1,\ldots,Y_r$ are closed subschemes of $V$, we have
	\begin{align}
		\lambda_{c_1Y_1\cap\ldots\cap c_rY_r,v}&=\min\{ c_1\lambda_{Y_1,v},\ldots,c_r\lambda_{Y_r,v}\},\label{cap}\\
		\lambda_{c_1Y_1+\cdots+c_rY_r,v}&=c_1\lambda_{Y_1,v}+\cdots+c_r\lambda_{Y_r,v}.\label{plus}
	\end{align}
	If $\phi:W\to V$ is a morphism of projective varieties with $\phi(W) \not \subset Y$, then up to $O(1)$,
	\begin{equation*}
		\lambda_{Y,v}(\phi(P))=\lambda_{\phi^*Y,v}(P),
	\end{equation*}
	\begin{equation*}
		\quad \forall P\in W(k)\setminus \phi^*Y.
	\end{equation*}
	Specifically, it is a subtle point that if $Y$ is associated with the ideal sheaf $\mathcal{I}_Y$, then $\phi^* Y$ represents the closed subscheme corresponding to the inverse image ideal sheaf $\phi^{-1} \mathcal{I}_Y \cdot \mathcal{O}_W$.
	When working with moving targets $D(\alpha)$, we write $\lambda_{D(\alpha),v}(\mathbf{x}(\alpha))$, understanding this is defined for $\alpha$ where the representative is valid.
	\subsection{Moving hypersurfaces/polynomials}
	A moving hypersurface $D$ in $\mathbb{P}^{M}(k)$ of degree $d$, indexed by $\Lambda$ (an infinite index set) is collection of hypersurfaces $\{D(\alpha)\}_{\alpha \in \Lambda}$ which are defined by the zero of homogeneous polynomials $\{Q(\alpha)\}_{\alpha \in \Lambda}$ in $k\left[x_{0}, \ldots, x_{M}\right]$ respectively. Then, we can write $$Q=\sum_{I \in \mathcal{T}_{d}} a_{I} x^{I},$$ 
	it can be referred to as moving homogeneous polynomial, where $a_{I}$ 's are functions from $\Lambda$ into $k$ having no common zeros points.\par
	Consider an infinite index $\Lambda$; a set $\mathcal{Q}:=\left\{D_{1}, \ldots, D_{q}\right\}$ of moving hypersurfaces in $\mathbb{P}^{M}(k)$, indexed by $\Lambda$, which are defined by the zero of moving homogeneous polynomials $\left\{Q_{1}, \ldots, Q_{q}\right\}$ in $k\left[x_{0}, \ldots, x_{M}\right]$ respectively, an arbitrary projective variety $V \subset \mathbb{P}^{M}(k)$ of dimension $n$ generated by the homogeneous ideal $\mathcal{I}(V)$. Marked as
	$$
	Q_{j}=\sum_{I \in \mathcal{T}_{d_{j}}} a_{j, I} x^{I}(j=1, \ldots, q), \text { where } d_{j}=\operatorname{deg} Q_{j} .
	$$
	The aforementioned settings are presumed to be in effect in the subsequent content.
	We will now provide definitions for several key terms:
    \begin{definition}[coherent]\label{def:coherent}
    	A subset $A \subseteq \Lambda$ is \emph{coherent} with respect to $\cQ$ if for every polynomial 
    	$$ P \in k[\{x_{j,I}\}_{j=1,\dots,q; I \in \mathcal{T}_{d_j}}] $$
    	that is homogeneous in the variables $x_{j,1}, \dots, x_{j,M_{d_j}}$ for each $j$, either
    	$P(a_{1,I_{1,1}}(\alpha), \dots) = 0$ for all but finitely many $\alpha \in A$, or for only finitely many $\alpha \in A$.
    \end{definition}
    
    By \cite[Prop.~2.3]{RV}, there exists an infinite coherent subset $A \subseteq \Lambda$. For such $A$, fix for each $j$ an index $I_j \in \mathcal{T}_{d_j}$ with $a_{j,I_j} \not\equiv 0$ on $A$. Then the ring 
    \[ \cR_{A,\cQ}^0 := \left\{ \frac{a}{b} \;\middle|\; a,b: A \to k,\; b(\alpha) \neq 0 \text{ for a.e. } \alpha \right\} \]
    contains the ratios $a_{j,I}/a_{j,I_j}$. The field of fractions of the subring generated by these ratios over $k$ is denoted $\cR_{A,\cQ}$.
    
    \begin{definition}[Special representatives]\label{def:spec-rep}
    	For $f \in \cR_{A,\cQ}$, a \emph{special representative} is a map $$\widehat{f}: \{\alpha \in A : b(\alpha) \neq 0\} \to k$$ for some expression $f = a/b$. For a polynomial $Q = \sum a_I x^I \in \cR_{A,\cQ}[x_0,\dots,x_M]$, a special representative $\widehat{Q}$ is defined coefficient-wise. It is well-defined at $\alpha$ if all coefficient representatives are.
    \end{definition}
	
	\begin{definition}[$V$-algebraically non-degenerate]
		A sequence of points $x=\left[x_{0}: \cdots: x_{M}\right]: \Lambda \rightarrow V$ is said to be $V$-algebraically non-degenerate with respect to $\mathcal{Q}$ if for each infinite coherent subset $A \subset \Lambda$ with respect to $\mathcal{Q}$, there is no homogeneous polynomial $Q \in \mathcal{R}_{A, \mathcal{Q}}\left[x_{0}, \ldots, x_{M}\right] \backslash \mathcal{I}_{A, \mathcal{Q}}(V)$ such that $$\widehat{Q}(\alpha)\left(x_{0}(\alpha), \ldots, x_{M}(\alpha)\right)=0$$ for all, but finitely many, $\alpha \in A$ for some (then for all) representative $\widehat{Q}$ of $Q$, where $\mathcal{I}_{A, \mathcal{Q}}(V)$ is the ideal in $\mathcal{R}_{A, \mathcal{Q}}\left[x_{0}, \ldots, x_{M}\right]$ genarated by $\mathcal{I}(V)$.
	\end{definition}
	\begin{remark}
		For each fixed $\alpha \in \Lambda$, $D(\alpha) = \mathrm{Proj}(k[x]/(Q(\alpha)))$ is a classical closed subscheme of $\mathbb{P}^M$. The intersection of finitely many such hypersurfaces is again a closed subscheme by standard algebraic geometry \cite{H}.
	\end{remark}
	According to the definitions and conclusions presented above, we can construct the following concept:
	\subsection{Family of moving targets}
	Let $A \subset \Lambda$ be an infinite subset and denote by $(A, a)$ each set-theoretic map $a: A \rightarrow k$. Denote by $\mathcal{R}_{A}^{0}$ the set of equivalence classes of pairs ( $C, a$ ), where $C \subset A$ is a subset with finite complement and $a: C \rightarrow k$ is a map, we can define the equivalent relation: if there exists $C \subset C_{1} \cap C_{2}$ such that $C$ has finite complement in $A$ and $\left.a_{1}\right|_{C}=\left.a_{2}\right|_{C}$, then we have $\left(C_{1}, a_{1}\right) \sim\left(C_{2}, a_{2}\right)$.

	Let $\mathbf{x}: \Lambda \rightarrow V \subset P^{M}(k)$ be a map. A map $(C, a) \in \mathcal{R}_{A}^{0}$ is called small with respect to $\mathbf{x}$ if and only if
	$$
	h(a(\alpha))=o(h(\mathbf{x}(\alpha))),
	$$
	for every $\epsilon>0$, it shows that there exists a subset $C_{\epsilon} \subset C$ with finite complement such that $h(a(\alpha)) \leq \epsilon h(\mathbf{x}(\alpha))$ for all $\alpha \in C_{\epsilon}$. Denote by $\mathcal{K}_{\mathbf{x}}$ the set of all such small maps. Then, $\mathcal{K}_{\mathbf{x}}$ is subring of $\mathcal{R}_{A}^{0}$. It is not an entire ring, however, if $(C, a) \in \mathcal{K}_{x}$ and $a(\alpha) \neq 0$ for all but finitely $\alpha \in C$, then we have $\left(C \backslash\{\alpha: a(\alpha)=0\}, \frac{1}{a}\right) \in \mathcal{K}_{\mathbf{x}}$. Denote by $\mathcal{C}_{\mathbf{x}}$ the set of all positive functions $g$ defined over $\Lambda$ outside a finite subset of $\Lambda$ such that
	$$
	\log ^{+}(g(\alpha))=o(h(\mathbf{x}(\alpha)))
	$$
	Then $\mathcal{C}_{\mathbf{x}}$ is a ring.\par Moreover, if $(C, a) \in \mathcal{K}_{\mathbf{x}}$, then for every $v \in M_{k}$, the function $\|a\|_{v}$ : $C \rightarrow \mathbb{R}^{+}$given by $\alpha \mapsto\|a(\alpha)\|_{v}$ belongs to $\mathcal{C}_{\mathbf{x}}$. Furthermore, if $(C, a) \in \mathcal{K}_{\mathbf{x}}$ and $a(\alpha) \neq 0$ for all but finitely $\alpha \in C$, the function $$g:\{\alpha \mid a(\alpha) \neq 0\} \mapsto \frac{1}{\|a(\alpha)\|_{v}}$$ also belongs to $\mathcal{C}_{\mathbf{x}}$.
	We also have the following Lemma (see \cite{STT}):
	\begin{lemma}\label{moving family} Let $A \subset \Lambda$ be coherent with respect to $\mathcal{Q}$, supposed they are in general position. Then for each $J \subset\{1, \ldots, q\}$ with $\sharp J=n+1$, there are functions $\gamma_{1, v}, \gamma_{2, v} \in \mathcal{C}_{\mathbf{x}}$ such that
		$$
		\gamma_{2, v}(\alpha)\|\mathbf{x}(\alpha)\|_{v}^{d} \leq \max _{j \in J}\left\|Q_{j}(\alpha)(\mathbf{x}(\alpha))\right\|_{v} \leq \gamma_{1, v}(\alpha)\|\mathbf{x}(\alpha)\|_{v}^{d}
		$$
		for all $\alpha \in A$ ouside finite subset and all $v \in S$.
	\end{lemma}
	\subsection{Subgeneral position}
		Now we give the definition of the positional relationship of hypersurfaces:
	\begin{definition}\label{D2.1}Let $V$ be a projective variety defined over a number field $k$ and let $D_{1}(\alpha), \ldots, D_{q}(\alpha)$ be $q$ moving hypersurfaces of $V.$ Suppose $m$ and $\kappa$ are positive integers satisfying $m\geq\dim V\geq \kappa.$\par
		
		The hypersurfaces $\{D_{1}(\alpha), \ldots, D_{q}(\alpha)\}$  are taken to be in $m$-subgeneral position with respect to $V$ if for any subset $I\subset\{1, \ldots, q\}$ with $\sharp I\leq m+1,$
		\begin{align*}
			\codim\left(\bigcap_{i\in I}D_{i}(\alpha)\cap V\right)\geq \dim V-(m-\sharp I).
		\end{align*}
		for all, but finitely many $\alpha \in \Lambda$.
		If add the condition of with index $\kappa$, it should be also restrained with
		\begin{align*}
			\codim\left(\bigcap_{i\in I} D_{i}(\alpha)\cap V\right)\geq\sharp I
		\end{align*}
		for each finite index set satisfied $\sharp I\leq \kappa$.
	\end{definition}
	
	\subsection{Seshadri constants}
	
	Let $V$ be a projective variety and $Y$  be a closed subscheme of $V,$ corresponding to a coherent sheaf of ideals $\mathcal{I}_{Y}.$ Consider the sheaf of graded algebras $\mathcal{S} =
	\oplus_{d \geq 0}\mathcal{I}_{Y}^{d}$, where $\mathcal{I}_{Y}^{d}$ is the $d$-th power of $\mathcal{I}_{Y},$ and $\mathcal{I}_{Y}^{0} = \mathcal{O}_{V}$ by convention. Then $\tilde{V} := \Proj \mathcal{S}$ is called the blowing-up of $V$ with respect to $\mathcal{I}_{Y},$
	in other words, the blowing-up of $V$ along $Y.$
	
	Let $\pi: \tilde{V} \to V$ be the blow-up of $V$ along $Y$. It is a fact that the inverse image ideal sheaf $\tilde{\mathcal{I}}_Y = \pi^{-1} \mathcal{I}_Y \cdot \mathcal{O}_{\tilde{V}}$ is an invertible sheaf on $\tilde{V}$.\par
	Now we have the following notion:
	\begin{definition}\label{Sesh}
		Let $\pi:\tilde{V}\to V$ be the blowing-up of $V$ along $Y$. Let $A$ be a nef Cartier divisor on $V$. The {\it Seshadri constant} $\epsilon_{Y,A}$ of $Y$ with respect to $A$ is the real number defined by
		\begin{align*}
			\epsilon_{Y,A}=\sup\{\gamma\in {\mathbb{Q}}^{\geq 0}\mid \pi^*A-\gamma E\text{ is $\mathbb{Q}$-nef}\},
		\end{align*}
		where $E$ is an effective Cartier divisor on $\tilde V$ whose associated invertible sheaf is the dual of $\pi^{-1}\II_Y \cdot \O_{\tilde V}$.
	\end{definition}
	\begin{remark}
	It is easy to verify that when $Y$ is a hypersurface with its degree to be $d$, then we have 	
	$$\epsilon_{Y,A}=\frac{1}{d}\in \mathbb{Q}$$
    associated to $A$, where $A$ is a hyperplane.
	\end{remark}
	
	\section{Introduction}
	The theory of Schmidt's subspaces of fixed hypersurfaces and divisors has developed rapidly since the late 1990s and early 21th century (\cite{CZ}, \cite{EF}, \cite{MR} as pioneer).\par
	In 1997, Ru-Vojta \cite{RV} established the following Schmidt subspace theorem for the case of moving hyperplanes in projective spaces based on the work \cite{MR97}. Which laid a solid foundation for the following development in the realm of moving targets.
	
	\begin{theorem} 
		Let $k$ be a number field and let $S \subset M_{k}$ be a finite set containing all archimedean places. Let $\Lambda$ be an infinite index set and let $\mathcal{H}:=\left\{H_{1}, \ldots, H_{q}\right\}$ be a set of moving hyperplanes in $\mathbb{P}^{M}(k)$, indexed by $\Lambda$. Let $\mathbf{x}=\left[x_{0}: \cdots: x_{M}\right]: \Lambda \rightarrow \mathbb{P}^{M}(k)$ be a sequence of points. Assume that
		(i) $\mathbf{x}$ is linearly nondegenerate with respect to $\mathcal{H}$, which mean, for each infinite coherent subset $A \subset \Lambda$ with respect to $\mathcal{H},\left.x_{0}\right|_{A}, \ldots,\left.x_{M}\right|_{A}$ are linearly independent over $\mathcal{R}_{A, \mathcal{H}}$,
		(ii) $h\left(H_{j}(\alpha)\right)=o(h(\mathbf{x}(\alpha)))$ for all $\alpha \in \Lambda$ and $j=1, \ldots, q$.
		
		Then, for any $\epsilon>0$, there exists an infinite index subset $A \subset \Lambda$ such that
		$$
		\log\prod_{v \in S} \max _{K} \prod_{J \in K} \left(\frac{\|\mathbf{x}(\alpha)\|_{v}\|H_{j}(\alpha)\|_{v}}{\|H_{j}(\alpha)(\mathbf{x}(\alpha))\|_{v}}\right) \leq(M+1+\epsilon) h(\mathbf{x}(\alpha))
		$$
		holds for all $\alpha \in A$. Here the maximum is taken over all subsets $K$ of $\{1, \ldots, q\}, \sharp K=$ $M+1$ such that $H_{j}(\alpha), j \in K$ are linearly independent over $k$ for each $\alpha \in \Lambda$.
	\end{theorem}
	Building upon the methods of Dethloff-Tan \cite{DT11}, the result was extended to moving hypersurfaces in projective spaces. Subsequently, Dethloff-Tan \cite{DT20} further extended Ru's result \cite{MR} to the moving target setting in 2019.\par
	In 2018, Quang \cite{Q18} established Schmidt's subspace theorem for moving hypersurfaces in subgeneral position. The following year, Xie-Cao \cite{XC} obtained the Second Main Theorem for moving hypersurfaces in subgeneral position with index. In 2020, Cao and Thin \cite{CT} proved a Schmidt subspace theorem for moving hypersurfaces in m-subgeneral position with index $\kappa$, which we state as follows:
	\begin{theorem}
		Let $k$ be a number field and let $S \subset M_{k}$ be a finite set containing all archimedean places. Let $\mathbf{x}=\left[x_{0}: \cdots: x_{M}\right]: \Lambda \rightarrow V$ be a sequence of points. Assume that
		(i) $\mathcal{Q}$ is in $m$-strictly subgeneral position with index $\kappa$ in $V$ and $\mathbf{x}$ is $V-$ algebraically nondegenerate with respect to $\mathcal{Q}$;
		(ii) $h\left(D_{j}(\alpha)\right)=o(h(\mathbf{x}(\alpha)))$ for all $\alpha \in \Lambda$ and $j=1, \ldots, q$ .
		Then, for any $\epsilon>0$, there exists an infinite index subset $A \subset \Lambda$ such that
		$$
		\log\prod_{v \in S} \prod_{j=1}^{q} \left(\frac{\|\mathbf{x}(\alpha)\|_{v}^{d_j}\|D_{j}(\alpha)\|_{v}}{\|D_{j}(\alpha)(\mathbf{x}(\alpha))\|_{v}}\right)^\frac{1}{d_j}
		$$
		$$
		\leq\left(\left(\frac{m-n}{\max \{1, \min \{m-n, \kappa\}\}}+1\right)(n+1)+\epsilon\right) h(\mathbf{x}(\alpha))
		$$
		holds for all $\alpha \in A$.
	\end{theorem}
	It is widely recognized that Nevanlinna theory and Diophantine Approximation are deeply connected, as demonstrated by the works of Osgood and Vojta. In particular, Vojta established a dictionary that establishes correspondences between the fundamental concepts of these two theories. Through this dictionary, the subspace theorem, along with its extension to the moving target problem in Diophantine Approximation, corresponds to the second main theorem in Nevanlinna theory. As the latest accurate estimate in Nevanlinna theory, Quang studied a more general case for meromorphic map \cite{Q22}, which considered the case of arbitrary families of hypersurfaces, not required to be in subgeneral position. It contains a notion of distributive constant. In \cite{CQN}, the authors have adjust this concept into moving targets case in Schmidt subspace theory, i.e.-$\Delta_{\mathcal{Q},V}$ of a family of moving hypersurface $\mathcal{Q}=\left\{D_{i}\right\}_{i=1}^{q}$ of $\mathbb{P}^{M}(k)$ in a subvariety $V \subset \mathbb{P}^{M}(k)$, where $V \not \subset \operatorname{supp} D_{i}(i=1, \ldots, q)$, as follows:
	$$
	\Delta_{\mathcal{Q},V}:=\max _{{\Gamma \subset\{1, \ldots, q\}}, \alpha \in A} \frac{\sharp \Gamma}{\codim\left(\bigcap_{j \in \Gamma} D_{j}(\alpha)\right) \cap V} .
	$$
	for some infinite index subset $A$ of $\Lambda$.
	
	\begin{theorem} Let $k$ be a number field, $S$ be a finite set of places of $k$ and let $V$ be an irreducible projective subvariety of $\mathbb{P}^{M}(k)$ of dimension $n$. Let $\mathcal{Q}:=\left\{D_{1}, \ldots, D_{q}\right\}$ be a set of moving hypersurfaces in $\mathbb{P}^{M}(k)$, indexed by $\Lambda$, with the distributive constant $\Delta_{\mathcal{Q}, V}$ in $V(\bar{k})$. Assume that $h\left(D_{i}(\alpha)\right)=o\left(h(\mathbf{x}(\alpha))\right.$, where $\mathbf{x}=\left[x_{0}: \cdots: x_{M}\right]: \Lambda \rightarrow V($ a collection of points), we also assume that $\mathbf{x}$ is non-degenerate over $\mathcal{Q}$. Then for each $\epsilon>0$, there exists an infinite index subset $A \subset \Lambda$ such that
		$$
		\log \prod_{v \in S} \prod_{i=1}^{q}\left(\frac{\|\mathbf{x}(\alpha)\|_{v}^{d_{i}} \cdot\left\|D_{i}(\alpha)\right\|_{v}}{\left\|D_{i}(\mathbf{x}(\alpha))\right\|_{v}}\right)^{\frac{1}{d_{i}}} \leq\left(\Delta_{\mathcal{Q}, V}(n+1)+\epsilon\right) h(\mathbf{x}(\alpha))
		$$
		holds for all $\alpha \in A$ and all $v \in S$.
		Where $\codim \varnothing=\infty$.
	\end{theorem}
	In accordance with the latest research findings and advancements in \cite{HL25}, the authors have constructed the following vital inequalities with consideration of weights for fixed closed subschemes (also contain the definition of hypersurfaces): 
	\begin{lemma}(\cite{HL25}, Lemma 3.1, reformulated in logarithm)\label{generalChebyshev}
		Set $b_1,\ldots, b_n,c_1,\ldots, c_n$ be nonnegative real numbers as sets of indicators. Let $i_0$ be the smallest one of the index $i$ such that $c_{i}\not = 0$ (assume that there exists at least one satisfied), then
		\begin{align*}
			\log\prod_{i=1}^n\Delta^{b_i}_i\geq \log\prod_{i=1}^n\Delta^{\left(\min_{i_0\leq j\leq n} \frac{\sum_{i=1}^jb_i}{\sum_{i=1}^jc_i}\right)c_i}_i.
		\end{align*}
	\end{lemma}
	\begin{corollary}
		\label{cor_elemineq}
		Set $b_1,\ldots, b_n,c_1,\ldots, c_n$ be nonnegative real numbers as sets of indicators. Assume that $b_1\not = 0$. Then
		\begin{align*}
			\log\prod_{i=1}^n\Delta^{\left(\max_{1\leq j\leq n} \frac{\sum_{i=1}^jc_i}{\sum_{i=1}^jb_i}\right)b_i}_i\geq \log\prod_{i=1}^n\Delta^{c_i}_i.
		\end{align*}
	\end{corollary}
	They have also paid more attention to the special properties of Weil functions corresponding to the intersection of closed subschemes, which brings up the theorem in the formula of
	\begin{theorem}[\cite{HL25}, Theorem 1.2, reformulated in hypersurfaces with degree $d_{i}$]
		Let $V$ be a  projective variety of dimension $n$  defined over a number field $k$, and let $S$ be a finite set of places of $k$. For each $v\in S$, let $D_{1},\ldots, D_{q}$ be hypersurfaces of $V$ defined over $k$, and let $c_{1},\ldots, c_{q}$ be nonnegative real numbers.  For a closed subset $W\subset V$ and $v\in S$, let
		\begin{align*}
			\alpha_v(W)=\sum_{\substack{i\\ W\subset \Supp D_{i}}}c_{i}.
		\end{align*}
		Then there exists a proper Zariski closed subset $Z$ of $V$ and $\epsilon>0$ such that
		\begin{align*}
			\log\prod_{v\in S}\prod_{i=1}^q\left(\frac{\|\mathbf{x}\|_{v}^{d_i}\|D_{i}\|_{v}}{\|D_{i}(\mathbf{x})\|_{v}}\right)^{c_{i}\cdot\frac{1}{d_{i}}} \leq  \left((n+1)\max_{\substack{v\in S\\ \varnothing\subsetneq W\subsetneq X}}\frac{\alpha_v(W)}{\codim W}+\epsilon\right)h(\mathbf{x})
		\end{align*}
		for all points $\mathbf{x}\in V(k)\setminus Z$.
	\end{theorem}
	This has provided insights and inspiration for our main result below, we have extended it to the case of moving targets with distinct weights:
	\begin{theorem}\label{thm:main} Let $k$ be a number field, $S$ be a finite set of places of $k$ and let $V$ be an irreducible projective subvariety of $\mathbb{P}^{M}(k)$ of dimension $n$. Let $\mathcal{Q}:=\left\{D_{1}, \ldots, D_{q}\right\}$ be a set of moving hypersurfaces in $\mathbb{P}^{M}(k)$, indexed by $\Lambda$. Assume that $h\left(D_{i}(\alpha)\right)=o\left(h(\mathbf{x}(\alpha))\right.$, where $\mathbf{x}=\left[x_{0}: \cdots: x_{M}\right]: \Lambda \rightarrow V($ a collection of points), we also assume that $\mathbf{x}$ is non-degenerate over $\mathcal{Q}$. Then for each $\epsilon>0$ and let $c_{1},\ldots, c_{q}$ be nonnegative real numbers, there exists an infinite index subset $A \subset \Lambda$ such that
		\begin{align*}
			\log \prod_{v \in S} \prod_{i=1}^{q}\left(\frac{\|\mathbf{x}(\alpha)\|_{v}^{d_i}\|D_{i}(\alpha)\|_{v}}{\|D_{i}(\alpha)(\mathbf{x}(\alpha))\|_{v}}\right)^{c_{i}\cdot\frac{1}{d_i}}
			\leq \left((n+1)\max_{\substack{v\in S\\ \varnothing\subsetneq W\subsetneq X}}\frac{\alpha_v(W)}{\codim W}+\epsilon\right) h(\mathbf{x}(\alpha))
		\end{align*}
		holds for all $\alpha \in A$ outside finite set.
		Where
		\begin{align*}
			\alpha_v(W)=\sum_{\substack{i\\ W\subset \Supp D_{i}(\alpha)}}c_{i}.
		\end{align*}
		with a closed subset $W\subset X$ and $v\in S$, $\codim \varnothing=\infty$.
	\end{theorem}
	In comparison with Theorem 2.3, under the assumption $c_{i}=1$ for all $i$, from the definitions and a straightforward observation that when taking the maximum over $W$, it suffices to consider the cases where $W$ is the intersection of a subset of hypersurfaces of $\mathcal{Q}$, we have
	\begin{remark}(recover the result in \cite{CQN})\label{2.8}
		\begin{align*}
			\Delta_{\mathcal{Q},V}=\max _{{\Gamma \subset\{1, \ldots, q\}}, \alpha \in A} \frac{\sharp \Gamma}{\codim\left(\bigcap_{j \in \Gamma} D_{j}(\alpha)\right) \cap V}=\max\left\{1, \max_{\varnothing\subsetneq W\subsetneq X}\frac{\alpha_v(W)}{\codim W}\right\}.
		\end{align*}
		for some infinite index subset $A$ of $\Lambda$.
	\end{remark}
	\textbf{Justification of the equality in Remark \ref{2.8}}: When $c_i=1$ for all $i$, the equality $\alpha_v(W)=\sharp\left\{i:W \subset \Supp D_i(\alpha)\right\} $ counts the number of hypersurfaces containing $W$. For $W=\bigcap_{j \in \Gamma}D_j(\alpha)\cap V$, we have $\alpha_v(W) \geq \sharp \Gamma$ since all $D_j(\alpha)$ with $j \in \Gamma$ contain $W$. Conversely, any hypersurface $D_i(\alpha)$ containing $W$ must satisfy $W \subset D_i(\alpha)$, so the maximum is achieved when $W$ is exactly the intersection of the hypersurfaces in some subset $\Gamma$. The factor 1 in the maximum accounts for the case when no proper intersection structure gives a larger value.
	
	As a result, as stated in \cite{HL25}, we can obtain a factor that is not only sharp but also optimal in certain cases, especially when specific intersection exist. This flexibility is essential when different hypersurfaces contribute asymmetrically to the approximation error, a scenario previously unaddressed in the moving target framework. \par
	\section{An auxiliary result in algebra}
    \subsection{The Inertia Form Method}\label{subsec:inertia}
    
    The following is the foundation of our approach. Let $V \subseteq \mathbb{P}^M$ be an irreducible projective variety of dimension $n$, defined by the homogeneous ideal $\mathcal{I}(V) \subseteq k[x_0, \ldots, x_M]$.
    
    For a positive integer $N$, let
    \[
    H_V(N) = \dim_k k[x_0, \ldots, x_M]_N / \mathcal{I}(V)_N
    \]
    be the Hilbert function of $V$. For $N \gg 0$, $H_V(N)$ is a polynomial in $N$ of degree $n$.
    
    The following theorem is due to Zariski \cite{Z} in the classical setting, and was adapted to moving hypersurfaces by Son-Tan-Thin \cite{STT}.
    
    \begin{theorem}[Inertia Form Theorem]\label{thm:inertia}
    	Let $D_1, \ldots, D_n$ be hypersurfaces in $\mathbb{P}^M$ of degrees $d_1, \ldots, d_n$, defined by polynomials $Q_1, \ldots, Q_n$. Let $V \subseteq \mathbb{P}^M$ be an irreducible variety of dimension $n$ with $\deg V = \Delta$. Let $N$ be a positive integer divisible by all $d_i$ and sufficiently large.
    	
    	Then there exist homogeneous polynomials
    	\[
    	\phi_1, \ldots, \phi_{H_V(N)} \in k[x_0, \ldots, x_M]_N
    	\]
    	whose residue classes form a basis of $k[x]_N / \mathcal{I}(V)_N$, such that
    	\begin{equation}\label{eq:inertia}
    		\prod_{\ell=1}^{H_V(N)} \phi_\ell \equiv (Q_1 \cdots Q_n)^{\frac{\Delta N^{n+1}}{d_1 \cdots d_n (n+1)!} - u(N)} \cdot P \pmod{\mathcal{I}(V)_N}
    	\end{equation}
    	where:
    	\begin{itemize}
    		\item $u(N) = O(N^n)$ is an error term;
    		\item $P$ is a homogeneous polynomial of degree $\frac{\Delta N^{n+1}}{(n+1)!} + O(N^n)$;
    		\item The congruence means the difference lies in $\mathcal{I}(V)_N$.
    	\end{itemize}
    \end{theorem}
    
    \begin{remark}
    	The exponents in \eqref{eq:inertia} are chosen so that both sides have the same degree, we have $$
    	N \cdot H_{V}(N)-\frac{n \cdot \Delta \cdot N^{n+1}}{(n+1)!}+u(N)=\frac{\Delta \cdot N^{n+1}}{(n+1)!}+O\left(N^{n}\right),
    	$$
    	and
    	\begin{align*}
    		H_{V}(N) & :=\operatorname{dim}_{k} \frac{k\left[x_{0}, \ldots, x_{M}\right]_{N}}{\mathcal{I}(V)_{N}} \\
    		& =\operatorname{dim}_{\bar{k}} \frac{\bar{k}\left[x_{0}, \ldots, x_{M}\right]_{N}}{\mathcal{I}(V(\bar{k}))_{N}}=\Delta \cdot \frac{N^{n}}{n!}+O\left(N^{n-1}\right) .
    	\end{align*}
    \end{remark}

	\section{Proof of the Main Theorem}\label{subsec:proof}
	
	The proof proceeds in several steps. The key innovation is to use the inertia form method to construct an auxiliary map $F$, to which we apply the Schmidt subspace theorem for moving hyperplanes align with the filtration inequality.
	
	\begin{proof}[Proof of Theorem \ref{thm:main}]
		By definition \ref{def:coherent}, choose an infinite coherent subset $A \subseteq \Lambda$. For each $j = 1, \ldots, q$, fix an index $I_j \in \mathcal{T}_{d_j}$ such that $a_{j,I_j}(\alpha) \not\equiv 0$ on $A$. Define the \emph{normalized coefficients} and \emph{normalized polynomials}:
		\[
		\tilde{a}_{j,I} := \frac{a_{j,I}}{a_{j,I_j}} \in \mathcal{R}_{A,\mathcal{Q}}, \qquad \tilde{Q}_j := \sum_{I \in \mathcal{T}_{d_j}} \tilde{a}_{j,I} x^I \in \mathcal{R}_{A,\mathcal{Q}}[x_0, \ldots, x_M].
		\]
		It is easy to verify that
		$$
		\log \frac{\|\mathbf{x}\|_{v}^{d}\left\|Q_{i}\right\|_{v}}{\|Q_{i}(\mathbf{x})\|_{v}}=\log \frac{\|\mathbf{x}\|_{v}^{d}\|\tilde{Q}_{i}\|_{v}}{\|\tilde{Q}_{i}(\mathbf{x})\|_{v}}.
		$$
		Note that $\tilde{Q}_j$ has at least one coefficient equal to $1$ (the coefficient of $x^{I_j}$). The Weil functions of moving hypersurface $\tilde{D}_j(\alpha)$ defined by $\tilde{Q}_j(\alpha)$ satisfy for each $j$:
		\begin{equation}\label{eq:weil-normalize}
			\lambda_{\tilde{Q}_j(\alpha),v}(\mathbf{x}(\alpha)) = \lambda_{\tilde{D}_j(\alpha),v}(\mathbf{x}(\alpha)) = \log \frac{\|\mathbf{x}(\alpha)\|_v^{d_j} \|\tilde{Q}_j(\alpha)\|_v}{\|\tilde{Q}_j(\alpha)(\mathbf{x}(\alpha))\|_v} + O(1),
		\end{equation}
		where the $O(1)$ term depends only on $v$ and the choice of $I_j$, not on $\alpha$.\par
		For each $v \in S$ and $\alpha \in A$, order the indices $\{1, \ldots, q\}$ as $\{i_1(v,\alpha), \ldots, i_q(v,\alpha)\}$ such that
		\[
		\lambda_{\tilde{Q}_{i_1(v,\alpha)}(\alpha),v}(\mathbf{x}(\alpha)) \geq \lambda_{\tilde{Q}_{i_2(v,\alpha)}(\alpha),v}(\mathbf{x}(\alpha)) \geq \cdots \geq \lambda_{\tilde{Q}_{i_q(v,\alpha)}(\alpha),v}(\mathbf{x}(\alpha)) \geq 0.
		\]
		In order to get our result we need to construct the following intersections of the normalized hypersurfaces,
		set $$\tilde D_{I_{(v,\alpha)}}(\alpha)=\bigcap_{i\in I}\tilde  D_{i(v,\alpha)}(\alpha)$$ 
		thus we have
		\begin{align*}
			\lambda_{\tilde Q_{i_{j}(v,\alpha)}(\alpha),v}(\mathbf{x}(\alpha))=\min_{j'\leq j}\lambda_{\tilde Q_{i_{j'}(v,\alpha)}(\alpha),v}(\mathbf{x}(\alpha))=\lambda_{\tilde Q_{I_{j}(v,\alpha)}(\alpha),v}(\mathbf{x}(\alpha)).
		\end{align*}
		Set $I_{{j}(v,\alpha)}=\{i_{1}(v,\alpha),\ldots, i_{j}(v,\alpha)\}$. Let $l_v(\alpha)$ be the largest index such that
		\begin{align*}
			\bigcap_{j=1}^{l_v(\alpha)}\tilde D_{i_{j}(v,\alpha)}(\alpha)\cap V(\bar{k})\neq \varnothing.
		\end{align*}
		In particular, for $j>l_v(\alpha)$, $\lambda_{\tilde Q_{i_{j}(v,\alpha)}(\alpha)}(\mathbf{x}(\alpha))$ is bounded by a constant independent of $\mathbf{x}(\alpha)$, set $d_i=d_{{i}(v,\alpha)}$, we have
		\begin{align*}
			\sum_{v \in S}\sum_{i=1}^{q} \frac{\lambda^{c_{i}}_{\tilde Q_{i}(\alpha),v}(\mathbf{x}(\alpha))}{d_{i}}\leq \sum_{v \in S}\sum_{j=1}^{l_v(\alpha)} \frac{\lambda^{c_{i_j(v,\alpha)}}_{\tilde Q_{I_{j}(v,\alpha)}(\alpha)}(\mathbf{x}(\alpha))}{d_{i_j}}. 
		\end{align*}
		For each $v \in S$ and $\alpha \in A$, define the \emph{codimension function} for the initial segment intersection:
		\[
		b_{j(v,\alpha)} := \min\left\{\mathrm{codim}\left(\bigcap_{s=1}^j \tilde{D}_{i_s(v,\alpha)}(\alpha) \cap V\right), n\right\}, \quad b_{0(v,\alpha)} := 0.
		\]
		
		Apply the weighted filtration inequality (Lemma \ref{generalChebyshev}) with:
		\begin{itemize}
			\item $\Delta_j = \exp\left(\frac{1}{d_{i_j}} \lambda_{\tilde{Q}_{i_j(v,\alpha)}(\alpha),v}(\mathbf{x}(\alpha))\right)$ for $j = 1, \ldots, q$;
			\item $b_1 = b_{1(v,\alpha)}-b_{0(v,\alpha)}$, $\cdots$, $b_{l_v(\alpha)} = b_{{l_v(\alpha)}(v,\alpha)}-b_{{l_v(\alpha)-1}(v,\alpha)}$
			\item $c_j = c_{i_j(v,\alpha)}$.
		\end{itemize}
		Using the telescoping property for the $b_i$ and seshadri constant property of hypersurface, we have
		\begin{equation}\label{cheb_application}
			\sum_{v\in S}\left(\max_{j} \frac{\sum_{j'=1} ^{j} c_{i_{j'}(v,\alpha)}}{b_{j(v,\alpha)}}\right) \sum_{s=1}^{l_v(\alpha)} (b_{s(v,\alpha)}-b_{s-1(v,\alpha)})\frac{\lambda_{\tilde Q_{I_{s}(v,\alpha)}}(\mathbf{x}(\alpha))}{d_{I_s}}
		\end{equation}
		
		\begin{equation*}
			\geq \sum_{v \in S}\sum_{j=1}^{l_v(\alpha)} \frac{\lambda^{c_{i_j(v,\alpha)}}_{\tilde Q_{I_{j}(v,\alpha)}(\alpha)}(\mathbf{x}(\alpha))}{d_{i_j}}.
		\end{equation*}
		The key observation is that we can form a new list such that $\tilde Q_{I_{s}(v,\alpha)}$ are in general position with respect to $V$. (with $\tilde Q_{I_{s}(v,\alpha)}$ repeated $b_{s(v,\alpha)}-b_{s-1(v,\alpha)}$ times and omitted if $b_{s(v,\alpha)}-b_{s-1(v,\alpha)}=0$, see \cite{HL25})\par 
		Set the above indicator set as $\{1,\dots,k\}$,
		for each $v\in S$, and $\alpha \in A$, there exist a subset $$J(v, \alpha)=\{j_1(v, \alpha), \dots, j_n(v, \alpha)\} \subset \{1,\dots,k\}$$ such that
		\begin{align*}
			0< ||\tilde{Q}_{j_1(v, \alpha)}(\alpha)(\mathbf{x}(\alpha))||_v\le \dots& \le ||\tilde{Q}_{j_n(v, \alpha)}(\alpha)(\mathbf{x}(\alpha))||_v\\
			&\le \min_{j\not \in \{j_1(v, \alpha), \dots, j_n(v, \alpha)\}} ||\tilde{Q}_{j}(\alpha)(\mathbf{x}(\alpha))||_v.
		\end{align*}
		By using the lemma \ref{moving family} we have
		\begin{align*}
			&\log \prod_{j=1}^{k}||\tilde{Q}_j(\alpha)(\mathbf{x}(\alpha))||_v\\
			&=\log \prod_{\beta_j\not \in J(v, \alpha)}||\tilde{Q}_{\beta_j}(\alpha)(\mathbf{x}(\alpha))||_v+\log \prod_{t=1}^{n}||\tilde{Q}_{j_t(v, \alpha)}(\alpha)(\mathbf{x}(\alpha))||_v\notag\\
			&\ge (k-n)d \log ||\mathbf{x}(\alpha)||_v-\log \overset{\sim}h_v(\mathbf{x}(\alpha))\notag
			+\log \prod_{t=1}^{n}||\tilde{Q}_{j_t(v, \alpha)}(\alpha)(\mathbf{x}(\alpha))||_v,
		\end{align*}
		where $ \overset{\sim}h_v=\prod(1+h_{\mu})$, $h_{\mu}$ runs over all the choices of $\gamma_{2, v}$, thus $ \overset{\sim}h_v \in \mathcal C_x.$\par
		Now we construct the auxiliary map $F$ via inertia forms (our ordering above dose not affect subsequent construction and results):
		
		Fix $N \gg 0$ divisible by all $d_j$. Consider the relative coordinate ring:
		\[
		R_N := \mathcal{R}_{A,\mathcal{Q}}[x_0, \ldots, x_M]_N / \mathcal{I}_{A,\mathcal{Q}}(V)_N.
		\]
		Let $H_V(N) = \dim_{\mathcal{R}_{A,\mathcal{Q}}} R_N$.
		
		For each choice of $n$ indices $J(v,\alpha) = \{j_1(v,\alpha), \ldots, j_n(v,\alpha)\}$, apply Theorem \ref{thm:inertia} to the normalized polynomials $\tilde{Q}_{j_1(v,\alpha)}, \ldots, \tilde{Q}_{j_n(v,\alpha)}$. There exist:
		\begin{itemize}
			\item A basis $\phi_1^{J(v,\alpha)}, \ldots, \phi_{H_V(N)}^{J(v,\alpha)}$ of $R_N$;
			\item A polynomial $P_{J(v,\alpha)}$ of degree $\frac{(\deg V) N^{n+1}}{(n+1)!} + O(N^n)$;
		\end{itemize}
		such that for all $\mathbf{x}(\alpha) \in V(k)$:
		\begin{equation}\label{eq:inertia-specialize}
			\prod_{\ell=1}^{H_V(N)} \phi_\ell^{J(v,\alpha)}(\alpha)(\mathbf{x}(\alpha)) = \left(\prod_{t=1}^n \tilde{Q}_{j_t(v,\alpha)}(\alpha)(\mathbf{x}(\alpha))\right)^{\frac{(\deg V) N^{n+1}}{d(n+1)!} - u(N)} \cdot P_{J(v,\alpha)}(\alpha)(\mathbf{x}(\alpha))
		\end{equation}
		where $d = d_{j_1} \cdots d_{j_n}$ and $u(N) = O(N^n)$.
		
		Now fix an arbitrary basis $\Phi_1, \ldots, \Phi_{H_V(N)}$ of $R_N$. For each $J(v,\alpha)$, there exist linear forms (they are linear independent in $\cR_{A,\cQ})$
		\[
		L_1^{J(v,\alpha)}, \ldots, L_{H_V(N)}^{J(v,\alpha)} \in \mathcal{R}_{A,\mathcal{Q}}[y_1, \ldots, y_{H_V(N)}]
		\]
		such that
		\[
		\phi_\ell^{J(v,\alpha)} \equiv L_\ell^{J(v,\alpha)}(\Phi_1, \ldots, \Phi_{H_V(N)}) \pmod{\mathcal{I}_{A,\mathcal{Q}}(V)}.
		\]
		We  write
		$$L_\ell^{J(v,\alpha)}(y_1, \dots, y_{H_V(N)})=\sum_{s=1}^{H_V(N)}g_{\ell r}y_s, \quad g_{\ell s}\in\mathcal R_{A, \cQ}.$$
		Since $L_1^{J(v,\alpha)}, \dots, L_{H_N(V)}^{J(v,\alpha)}$ are linear independent over $\mathcal R_{A, \cQ},$   we have $\det (h_{\ell s})\ne 0\in \mathcal R_{A,\cQ}.$
		Thus, due to cohenrent property  of $A,$
		$\det(h_{\ell s})(\beta)\ne 0$  for all $\beta\in A$, outside a finite subset of $A.$ By passing to an infinite subset if necessary, we may assume that $L_1^{J(v,\alpha)}(\beta),\dots,L_{H_N(V)}^{J(v,\alpha)}(\beta)$ are lineraly independent over $k$ for all $\beta\in A.$\par
		Define the map:
		\[
		F: A \to \mathbb{P}^{H_V(N)-1}(k), \qquad F(\alpha) = [\Phi_1(\mathbf{x}(\alpha)), \ldots, \Phi_{H_V(N)}(\mathbf{x}(\alpha))].
		\]
		We claim that $F$ is linearly nondegenerate with respect to $\cL.$ Indeed, ortherwise, then there is a linear form $L\in\cR_{B, \cL}[y_1,\dots, y_{H_N(V)}]$ for some infinite coherent subset $B\subset A,$ such that $L(F)|_B\equiv 0$ in $B,$ which contradicts to the assumption that $x$ is algebraically nondegenerate with respect to $\cQ.$
		
		On the other hand, it is easy to see that there exist $h_{J(v,\alpha)} \in \mathcal{C}_{\mathbf{x}}$ such that
		\begin{align*}
			& \|P_{J(v,\alpha)}(\alpha)(\mathbf{x}(\alpha))\|_{v} 
			\leq  \|\mathbf{x}(\alpha)\|_{v}^{\operatorname{deg} P_{J}} h_{J(v,\alpha)}(\alpha)=\|\mathbf{x}(\alpha)\|_{v}^{\frac{\operatorname{deg} V \cdot N^{n+1}}{(n+1)!}+v(N)} h_{J(v,\alpha)}(\alpha).
		\end{align*}
		Therefore, we have
		\begin{align*}
			& \log \prod_{\ell=1}^{H_{V}(N)}\|\phi_{\ell}^{J(v,\alpha)}(\alpha)(\mathbf{x}(\alpha))\|_{v} \\
			\leq & \left(\frac{\operatorname{deg} V \cdot N^{n+1}}{d(n+1)!}-u(N)\right) \cdot \log \|\prod_{t=1}^n \tilde{Q}_{j_t(v,\alpha)}(\alpha)(\mathbf{x}(\alpha))\|_{v} \\
			& +\log ^{+} h_{J(v,\alpha)}(\alpha)+\left(\frac{\operatorname{deg} V \cdot N^{n+1}}{(n+1)!}+v(N)\right) \log \|\mathbf{x}(\alpha)\|_{v}.
		\end{align*}
		This implies that there are functions $\zeta_{1}(N), \zeta_{2}(N) \leq O\left(\frac{1}{N}\right)$ such that
		\begin{align*}
			& \log \|\prod_{t=1}^n \tilde{Q}_{j_t(v,\alpha)}(\alpha)(\mathbf{x}(\alpha))\|_{v} \\
			\geq & \left(\frac{d(n+1)!}{\operatorname{deg} V \cdot N^{n+1}}+\zeta_{1}(N)\right) \cdot \log \prod_{\ell=1}^{H_{V}(N)}\|\phi_{\ell}^{J(v,\alpha)}(\alpha)(\mathbf{x}(\alpha))\|_{v} \\
			& -\left(\frac{d(n+1)!}{\operatorname{deg} V \cdot N^{n+1}}+\zeta_{1}(N)\right) \log ^{+} h_{J(v,\alpha)}(\alpha)-\left(d+\zeta_{2}(N)\right) \log \|\mathbf{x}(\alpha)\|_{v}.
		\end{align*}
	
		Also, we have the following lemmas to estimate the heights and linear form:
		\begin{lemma}\label{lem:height-F}
			The height of $F$ satisfies:
			\[
			h(F(\alpha)) \leq N \cdot h(\mathbf{x}(\alpha)) + o(h(\mathbf{x}(\alpha))).
			\]
		\end{lemma}
		
		\begin{proof}
			Each $\Phi_\ell$ has degree $N$ in $\mathbf{x}$, so $\|\Phi_\ell(\mathbf{x}(\alpha))\|_v \leq \|\mathbf{x}(\alpha)\|_v^N \cdot \|\Phi_\ell\|_v$. Summing over $v \in M_k$:
			\[
			h(F(\alpha)) = \sum_{v \in M_k} \log \max_\ell \|\Phi_\ell(\mathbf{x}(\alpha))\|_v \leq N \cdot h(\mathbf{x}(\alpha)) + \sum_{v \in M_k} \log \max_\ell \|\Phi_\ell\|_v.
			\]
			The second term is bounded independent of $\alpha$. \end{proof}
		
		\begin{lemma}\label{lem:height-L}
			For each $J$ and $\ell$, the height of the linear form $L_\ell^J$ satisfies:
			\[
			h(L_\ell^{J(v,\alpha)}(\alpha)) = o(h(\mathbf{x}(\alpha))).
			\]
		\end{lemma}
		
		\begin{proof}
			The coefficients of $L_\ell^{J(v,\alpha)}$ are polynomials in the normalized coefficients $\tilde{a}_{j,I}$ with coefficients in $k$. Since $h(\tilde{a}_{j,I}(\alpha)) \leq 2h(Q_j(\alpha)) + O(1) = o(h(\mathbf{x}(\alpha)))$ by the small height hypothesis, the claim follows. \end{proof}

		Apply the Schmidt subspace theorem for moving hyperplanes (Ru-Vojta \cite{RV}, or Cao-Quang-Thin \cite{CQN} for the weighted version) to the map $F$ and the family of moving hyperplanes $\{L_\ell^{J(v,\alpha)}\}$ at each $v$ and $\alpha$:
		
		For any $\epsilon > 0$, there exists an infinite subset $A' \subseteq A$ such that for all $\alpha \in A'$:
		\begin{equation}\label{eq:sst-F}
			\sum_{v \in S} \sum_{\ell=1}^{H_V(N)} \log \frac{\|F(\alpha)\|_v \|L_\ell^{J(v,\alpha)}(\alpha)\|_v}{\|L_\ell^{J(v,\alpha)}(\alpha)(F(\alpha))\|_v} \leq (H_V(N) + \epsilon) h(F(\alpha)).
		\end{equation}
		From the inertia form equation \eqref{eq:inertia-specialize}, taking logarithms and using the definitions of $F$ and $L_\ell^J$:
		\begin{multline}\label{eq:translate}
			\sum_{v\in S}\sum_{j=1}^{k}\log \dfrac{||\mathbf{x}(\alpha)||_v^{d}}{||\tilde{Q}_j(\alpha)(\mathbf{x}(\alpha))||_v}
			\le (n+1)d \sum_{v\in S}\log ||\mathbf{x}(\alpha)||_v+\zeta_2(N)\sum_{v\in S}\log ||\mathbf{x}(\alpha))||_v\notag\\
			+\Big(\dfrac{d(n+1)!}{\deg V \cdot N^{n+1}}-\frac{\zeta_1(N)}{N^{n+1}}\Big)\sum_{v\in S}\log \prod_{\ell=1}^{H_N(V)}\dfrac{||F(\alpha)||_v||L_\ell^{J(v,\alpha)}(\alpha)||_v}{||L_\ell^{J(v, \alpha)}(\alpha)(\mathbf{x}(\alpha))||_v}\notag\\
			-H_V(N)\Big(\dfrac{d(n+1)!}{\deg V \cdot N^{n+1}}-\frac{\zeta_1(N)}{N^{n+1}}\Big)\sum_{v\in S}\log ||F(\alpha)||_v\notag
			+o(h(\mathbf{x}(\alpha))).
		\end{multline}
		
		Rearranging and using \eqref{eq:sst-F}, Lemma \ref{lem:height-F}, and Lemma \ref{lem:height-L} we can obtain
		\begin{align}
			\sum_{v\in S}\sum_{j=1}^{k}\log \dfrac{||\mathbf{x}(\alpha)||_v^{d}||\tilde{Q}_j(\alpha)||_v}{||\tilde{Q}_j(\alpha)(\mathbf{x}(\alpha))||_v}&\le  (n+1)dh(\mathbf{x}(\alpha))+\zeta_2(N)h(\mathbf{x}(\alpha))\notag\\
			&+\epsilon\Big(\dfrac{d(n+1)!}{\deg V \cdot N^{n+1}}-\frac{\zeta_1(N)}{N^{n+1}}\Big)h(F(\alpha))+o(h(\mathbf{x}(\alpha))).
		\end{align}
		For moving hypersurfaces in general position, the inequality implies
		\begin{equation}\label{eq:estimate-n}
			\sum_{v \in S} \sum_{i=1}^k \frac{\lambda_{\tilde{Q}_{i}(\alpha),v}(\mathbf{x}(\alpha))}{d_i} \leq (n+1 + \epsilon) h(\mathbf{x}(\alpha)) + o(h(\mathbf{x}(\alpha))).
		\end{equation}

		Applying Lemma \ref{generalChebyshev} and using (\ref{cheb_application}), \eqref{eq:estimate-n}:
		\[
			\sum_{v\in S}\left(\max_{j} \frac{\sum_{j'=1} ^{j} c_{i_{j'}(v,\alpha)}}{b_{j(v,\alpha)}}\right) \sum_{s=1}^{l_v(\alpha)} (b_{s(v,\alpha)}-b_{s-1(v,\alpha)})\frac{\lambda_{\tilde Q_{I_{s}(v,\alpha)}}(\mathbf{x}(\alpha))}{d_{I_s}}\]
		\[\leq \max_{1 \leq j \leq q} \frac{\sum_{s=1}^j c_{i_s(v,\alpha)}}{b_{j(v,\alpha)}} \cdot (n+1+\epsilon) h(\mathbf{x}(\alpha)) + o(h(\mathbf{x}(\alpha))).
		\]
		
		For the initial segment $W_{j(v,\alpha)} = \bigcap_{s=1}^j \tilde{D}_{i_s(v,\alpha)}(\alpha) \cap V$, we have:
		\[
		\frac{\sum_{s=1}^j c_{i_s(v,\alpha)}}{b_{j(v,\alpha)}} = \frac{\alpha_v(W_{j(v,\alpha)})}{\mathrm{codim}\,W_{j(v,\alpha)}} \leq \max_{\varnothing \subsetneq W \subsetneq V} \frac{\alpha_v(W)}{\mathrm{codim}\,W}.
		\]
		
		Taking the maximum over $v \in S$ and the finitely many possible orderings of $\{1, \ldots, q\}$, and summing over $v \in S$:
		\[
		\sum_{v \in S} \sum_{j=1}^q c_j \frac{\lambda_{\tilde{Q}_j(\alpha),v}(\mathbf{x}(\alpha))}{d_j} \leq \left((n+1) \max_{\substack{v \in S \\ \varnothing \subsetneq W \subsetneq V}} \frac{\alpha_v(W)}{\mathrm{codim}\,W} + \epsilon\right) h(\mathbf{x}(\alpha))
		\]
		for all $\alpha \in A'$ outside a finite union of proper closed subsets, which is avoided by the non-degeneracy hypothesis.
		
		Exponentiating and using \eqref{eq:weil-normalize} yields the desired inequality. This completes the proof of Theorem \ref{thm:main}
		\begin{align*}
			\log \prod_{v \in S} \prod_{i=1}^{q}\left(\frac{\|\mathbf{x}(\alpha)\|_{v}^{d_i}\|Q_{i}(\alpha)\|_{v}}{\|Q_{i}(\alpha)(\mathbf{x}(\alpha))\|_{v}}\right)^{c_{i}\cdot\frac{1}{d_i}}
			\leq \left((n+1)\max_{\substack{v\in S\\ \varnothing\subsetneq W\subsetneq X}}\frac{\alpha_v(W)}{\codim W}+\epsilon\right) h(\mathbf{x}(\alpha))
		\end{align*}
		holds for all $\alpha \in A$ outside finite set.
	\end{proof}

	\vskip0.2cm
	{\footnotesize \noindent
		{\sc GuanHeng Zhao}\\
		ShanDong University, People's Republic of China.\\
		\textit{E-mail}: 202200700222@mail.sdu.edu.cn}\\
	{\footnotesize \noindent
		{\sc YuXi Li}\\
		University of Macau, People's Republic of China.\\
		\textit{E-mail}: mc55457@um.edu.mo}
\end{document}